\documentclass[a4,10pt]{amsart}
\usepackage{amssymb,amstext,amsmath,amscd,amsthm,amsfonts,enumerate,graphicx,latexsym,color,cite}
\newtheorem{thm}{Theorem}[section]
\newtheorem{cor}[thm]{Corollary}

\newtheorem{prop}[thm]{Proposition}

\newtheorem*{thm*}{Theorem}

\theoremstyle{definition}
\newtheorem{dfn}[thm]{Definition}
\newtheorem{rem}[thm]{Remark}
\newtheorem{ques}[thm]{Question}

\newtheorem{ex}[thm]{Example}

\theoremstyle{remark}
\newtheorem*{conv}{Convention}
\newtheorem{claim}{Claim}
\newtheorem*{claim*}{Claim}

\renewcommand{\qedsymbol}{$\blacksquare$}
\numberwithin{equation}{thm}
\def\xx{\text{{\boldmath$x$}}}
\def\yy{\text{{\boldmath$y$}}}
\def\Ker{\operatorname{\mathsf{Ker}}}

\def\syz{\mathsf{\Omega}}
\def\tr{\mathsf{Tr}}
\def\li{\mathbb{L}}
\def\x{\mathsf{X}}
\def\link{\mathsf{L}}

\def\Ext{\operatorname{\mathsf{Ext}}}
\def\m{\mathfrak{m}}
\def\cm{\operatorname{\mathsf{CM}}}

\def\mcm{\operatorname{\mathsf{MCM}}}
\def\lcm{\operatorname{\mathsf{\underline{MCM}}}}
\def\ucm{\operatorname{\mathsf{\overline{MCM}}}}
\def\depth{\operatorname{\mathsf{depth}}}
\def\dim{\operatorname{\mathsf{dim}}}
\def\mod{\operatorname{\mathsf{mod}}}

\def\perf{\operatorname{\mathsf{perf}}}
\def\grade{\operatorname{\mathsf{grade}}}
\def\codim{\operatorname{\mathsf{codim}}}
\def\pd{\operatorname{\mathsf{pd}}}
\def\X{\mathcal{X}}
\def\ann{\operatorname{\mathsf{ann}}}
\def\Hom{\operatorname{\mathsf{Hom}}}
\def\scong{\cong_{\rm st}}

\def\Z{\mathbb{Z}}
\def\C{\mathbb{C}}
\def\p{\mathfrak{p}}
\def\q{\mathfrak{q}}
\def\n{\mathfrak{n}}
\def\soc{\operatorname{\mathsf{soc}}}
\def\ch{\operatorname{\mathsf{char}}}
\begin{document}
\allowdisplaybreaks
\title[Perfect linkage of Cohen--Macaulay modules]{Perfect linkage of Cohen--Macaulay modules\\
over Cohen--Macaulay rings}
\author{Kei-ichiro Iima}
\address{Department of Liberal Studies, National Institute of Technology, Nara College , 22 Yata-cho, Yamatokoriyama, Nara 639-1080, Japan}
\email{iima@libe.nara-k.ac.jp}
\author{Ryo Takahashi}
\address{Graduate School of Mathematics, Nagoya University, Furocho, Chikusaku, Nagoya, Aichi 464-8602, Japan}
\email{takahashi@math.nagoya-u.ac.jp}
\urladdr{http://www.math.nagoya-u.ac.jp/~takahashi/}
\thanks{RT was partly supported by JSPS Grant-in-Aid for Scientific Research (C) 25400038}
\dedicatory{Dedicated to Professor Yuji Yoshino on the occasion of his sixtieth birthday}
\begin{abstract}
In this paper, we introduce and study the notion of linkage by perfect modules, which we call perfect linkage, for Cohen--Macaulay modules over Cohen--Macaulay local rings.
We explore perfect linkage in connection with syzygies, maximal Cohen--Macaulay approximations and Yoshino--Isogawa linkage.
We recover a theorem of Yoshino and Isogawa, and analyze the structure of double perfect linkage.
Moreover, we establish a criterion for two Cohen--Macaulay modules of codimension one to be perfectly linked, and apply it to the classical linkage theory for ideals.
We also construct various examples of linkage of modules and ideals.
\end{abstract}
\maketitle
\section{Introduction}

The theory of linkage (to be precise, CI-linkage) of ideals in commutative algebra has been established by Peskine and Szpiro \cite{PS} in the 1970s, and it then has further been studied by Huneke, Migliore, Nagel, Schenzel, Ulrich \cite{H, HU1, HU2, HU3, M, N1, Sc} and many others.
One of the central studies about linkage of ideals targets ideals linked to complete intersection ideals, so-called licci ideals, and it has turned out that in a regular ring Cohen--Macaulay ideals of codimension $2$ and Gorenstein ideals of codimension $3$ are licci \cite{PS, W}.
These are the only classes known to be licci, while several results on licci monomial ideals have been obtained \cite{HU4, KTY, MN}.
On the other hand, the structure of linkage classes, especially even linkage classes, has been deeply investigated by Bolondi, Kustin, Migriore, Miller, Nagel, Nollet and Rao \cite{BM, KM, N1, No, R}. 

The classical linkage theory thus deals only with ideals, but it has also been extended to modules.
The notion of linkage of modules has been introduced and studied by Strooker \cite{St} and Yoshino and Isogawa \cite{YI} for Cohen--Macaulay modules. 
Martsinkovsky and Strooker \cite{MS} have extended the classical linkage without any restrictions. 
Nagel \cite{N2} has introduced a concept of module liaison that extends Gorenstein liaison of ideals and provided an equivalence relation among unmixed modules over a Gorenstein ring.
Recently they have also been investigated by many authors; see \cite{DGHS, DS1, DS2, Ni}.

In this paper, we introduce the notion of perfect linkage.
This is a new notion of linkage for modules centering on perfect modules, and includes the concept of linkage due to Yoshino and Isogawa, which we call Yoshino--Isogawa linkage.
For simplicity, let $R$ be a Gorenstein local ring.
Let $M$ be a Cohen--Macaulay $R$-module of codimension $r$.
We call a homomorphism $f:P\to M$ a {\em perfect morphism} of $M$ if $f$ is surjective and $P$ is a perfect $R$-module of codimension $r$.
Perfect morphisms always exist.
For a perfect morphism $f$ of $M$, we define the {\em perfect link} $\li_fM$ of $M$ with respect to $f$ by $\li_fM=\Ext_R^r(\Ker f,R)$.
Then $\li_fM$ is again a Cohen--Macaulay $R$-module of codimension $r$, and there exists a perfect morphism $g$ of $\li_fM$ such that $\li_g\li_fM\cong M$.
Two Cohen--Macaulay $R$-modules $M$ and $N$ of codimension $r$ are called {\em perfectly linked} if there exists a perfect morphism $f$ of $M$ such that $\li_fM\cong N$.
We should mention that there exists only one perfect linkage class over a regular local ring, which is completely different from the classical linkage theory.
Thus we mainly explore perfect linkage over a singular local ring.

For each $R$-module $M$, let $\syz M,\tr M,\x M,\link M$ denote the syzygy, the transpose, the maximal Cohen--Macaulay approximation and the Yoshino--Isogawa link of $M$, respectively.
(The Yoshino--Isogawa link is defined only for a maximal Cohen--Macaulay module.)
Our first main result analyzes the structure of the maximal Cohen--Macaulay approximations of perfect linkage, and gives the relationship of perfect linkage with Yoshino--Isogawa linkage.

\begin{thm}\label{1}
Let $M$ be a Cohen--Macaulay $R$-module of codimension $r$, and let $f:P\to M$ be a perfect morphism.
Then $\syz^rf:\syz^rP\to\syz^rM$ is a perfect morphism (up to free summands), and there is an exact sequence
$$
0 \to Z \to \li_{\syz^rf}(\syz^rM) \to \li_fM \to 0
$$
of $R$-modules, where $\li_{\syz^rf}(\syz^rM)$ is maximal Cohen--Macaulay and $Z$ has projective dimension $r-1$.
In particular, there are stable equivalences
$$
\x\,\li_fM\scong\li_{\syz^rf}(\syz^rM)\scong\link\,\syz^rM,\qquad\syz^r\li_fM\scong\syz^{r+1}\tr\syz^rM.
$$
\end{thm}

This theorem immediately recovers the main result of Yoshino and Isogawa.
Moreover, it yields the following result on double perfect linkage of Cohen--Macaulay modules.

\begin{thm}\label{2}
Let $M,N$ be Cohen--Macaulay $R$-modules of codimension $r$.
\begin{enumerate}[\rm(1)]
\item
Let $f,g$ be perfect morphisms of $M,N$ respectively.
Consider the following four conditions:
$$
{\rm(a)}\ \li_fM\cong\li_gN,\quad
{\rm(b)}\ \Ker f\cong\Ker g,\quad
{\rm(c)}\ \syz^rM\scong\syz^rN,\quad
{\rm(d)}\ \syz^r\li_fM\scong\syz^r\li_gN.
$$
Then the implications ${\rm(a)}\Leftrightarrow{\rm(b)}\Rightarrow{\rm(c)}\Leftrightarrow{\rm(d)}$ hold.
\item
Suppose that $r\le1$ and that {\rm(c)} is satisfied.
Then {\rm(b)} is also satisfied for some perfect morphisms $f,g$ of $M,N$ respectively.
\end{enumerate}
\end{thm}

This theorem especially characterizes the structure of the double perfect linkage of Cohen--Macaulay modules of codimension one.
Using this result, we can obtain a criterion for two Cohen--Macaulay modules of codimension one to be perfectly linked as follows, where $(-)^*=\Hom_R(-,R)$.

\begin{thm}\label{3}
Let $M,N$ be Cohen--Macaulay $R$-modules of codimension one.
\begin{enumerate}[\rm(1)]
\item
$M$ is doubly perfectly linked to $N$ if and only if $\syz M\scong\syz N$.
\item
$M$ is triply perfectly linked to $N$ if and only if $(\syz M)^*\scong\syz N$.
\item
If $M$ and $N$ are perfectly linked, then they are either doubly or triply perfectly linked.
\end{enumerate}
\end{thm}

Making use of the above theorems, we actually construct a lot of examples of perfect linkage of modules and classical CI-linkage of ideals, together with various observations on them.

The organization of this paper is as follows.
In Section 2, we make the precise definition of a perfect link, and state several basic properties.
In Section 3, we investigate the relationships of perfect links with syzygies, transposes, maximal Cohen--Macaulay approximations and Yoshino--Isogawa links.
Theorems \ref{1} and \ref{2} are proved in this section.
In Section 4, we apply our results to classical CI-linkage to investigate when two Cohen--Macaulay ideals $I,J$ of codimention $r$ to be CI-linked, and then give a proof of Theorem \ref{3}.
Various examples and observations about perfect linkage of modules and CI-linkage of ideals are also stated in this section.

\begin{conv}
Throughout the present paper, we assume that all rings are commutative and noetherian, and all modules are finitely generated.
Let $(R,\m,k)$ be a $d$-dimensional Cohen--Macaulay local ring with canonical module $\omega$.
The notation $(-)^*$ stands for the $R$-dual functor $\Hom_R(-,R)$.
We denote by $\mod R$ the category of finitely generated $R$-modules.
All subcategories are assumed to be full and closed under isomorphism.
We adopt the convention that the zero $R$-module $0$ satisfies $\dim0=-\infty$, $\depth0=\infty$ and $\pd0=-\infty$.
Cohen--Macaulay and maximal Cohen--Macaulay are abbreviated to CM and MCM, respectively.
\end{conv}

\section{Perfect morphisms and links}

In this section we make the definition of perfect links of CM modules over $R$, and state several basic properties of them.

Let $M$ be an $R$-module.
The {\em codimension} $\codim M$ of $M$ is defined by $\codim M=\dim R-\dim M$.
The {\em grade} $\grade M$ of $M$ is defined as the infimum of the nonnegative integers $n$ with $\Ext_R^n(M,R)\ne0$.
This is equal to the maximum of the lengths of $R$-sequences annihilating $M$.
The equality $\codim M=\grade M$ holds since $R$ is CM, and $\codim0=\grade0=\infty$ by definition.
We say that $M$ is {\em perfect} if $\pd M=\grade M$.
Note that $M$ is perfect if and only if $M$ is a CM $R$-module of finite projective dimension.
In this paper we mainly deal with the following three subcategories of $\mod R$.
\begin{align*}
\cm^r(R)&=\{\text{CM $R$-modules of codimension $r$}\}\cup\{0\}\\
&=\{M\in\mod R\mid\depth M\ge d-r\ge\dim M\}=\{M\in\mod R\mid\Ext^i(M,\omega)=0\text{ for all }i\ne r\},\\
\mcm(R)&=\{\text{MCM $R$-modules}\}\cup\{0\}=\{M\in\mod R\mid\depth M\ge d\}=\cm^0(R),\\
\perf^r(R)&=\{\text{perfect $R$-modules of codimension $r$}\}\cup\{0\}\\
&=\{M\in\mod R\mid\pd M\le\grade M\}=\{M\in\cm^r(R)\mid\pd M<\infty\}.
\end{align*}

CM modules of codimension $r$ and MCM modules over a quotient ring by a regular sequence of length $r$ are closely related with each other as follows.

\begin{prop}\label{cmrmcm}
\begin{enumerate}[\rm(1)]
\item
Let $M$ be an $R$-module, and $\xx=x_1,\dots,x_r$ an $R$-sequence annihilating $M$.
Then $M$ is a CM $R$-module of codimension $r$ if and only if $M$ is a MCM $R/\xx R$-module.
\item
Let $M_1,\dots,M_n$ be CM $R$-modules of codimension $r$.
Then there exists an $R$-sequence $\xx=x_1,\dots,x_r$ such that $M_1,\dots,M_n$ are MCM $R/\xx R$-modules.
\end{enumerate}
\end{prop}

\begin{proof}
(1) In both cases one has $\depth M=\dim M=\dim R/\xx R=d-r$.

(2) For each $i$ the module $M_i$ satisfies $\grade M_i=\codim M_i=r$.
By \cite[Proposition 1.2.10(c)]{BH} the ideal $\bigcap_{i=1}^n\ann M_i$ contains an $R$-sequence of length $r$, say, $\xx=x_1,\dots,x_r$.
The assertion follows from (1).
\end{proof}

We say that a subcategory $\X$ of $\mod R$ is {\em closed under extensions} (respectively, {\em closed under kernels of epimorphisms}) if for an exact sequence $0 \to L \to M \to N \to 0$ in $\mod R$ with $L,N\in\X$ (respectively, $M,N\in\X$) one has $M\in\X$ (respectively, $L\in\X$).
Let us investigate the structure of the subcategories $\cm^r(R)$ and $\perf^r(R)$.

\begin{prop}\label{qres}
\begin{enumerate}[\rm(1)]
\item
The subcategory $\cm^r(R)$ of $\mod R$ contains $\perf^r(R)$, and is closed under direct summands, extensions and kernels of epimorphisms.
\item
The functor $\Ext_R^r(-,\omega)$ induces a duality of $\cm^r(R)$.
\item
If $R$ is Gorenstein, then the functor $\Ext_R^r(-,R)$ induces a duality of $\perf^r(R)$.
\end{enumerate}
\end{prop}

\begin{proof}
(1) It is clear that $\cm^r(R)$ contains $\perf^r(R)$.
In general, over a CM local ring $A$ the subcategory $\mcm(A)$ of $\mod A$ is {\em resolving} (i.e., it contains the free $A$-modules and is closed under direct summands, extensions and kernels of epimorphisms).
Also, if $0\to L\to M\to N\to0$ is an exact sequence of $R$-modules and $\xx=x_1,\dots,x_r$ is an $R$-sequence in $\ann L\cap\ann N$, then $\xx^2=x_1^2,\dots,x_r^2$ is an $R$-sequence in $\ann L\cap\ann M\cap\ann N$.
From these facts and Proposition \ref{cmrmcm} we observe that $\cm^r(R)$ is closed under direct summands, extensions and kernels of epimorphisms.

(2) Let $M$ be a module in $\cm^r(R)$.
Proposition \ref{cmrmcm}(2) implies that $M$ is a MCM $R/\xx R$-module for some $R$-sequence $\xx=x_1,\dots,x_r$.
Note that $\Ext_R^r(M,\omega)$ is isomorphic to $\Hom_{R/\xx R}(M,\omega/\xx\omega)$.
As $\omega/\xx\omega$ is a canonical module of $R/\xx R$, the functor $\Hom_{R/\xx R}(-,\omega/\xx\omega)$ makes a duality of $\mcm(R/\xx R)$.
The assertion now follows by using Proposition \ref{cmrmcm}(1).

(3) Let $M$ be a perfect $R$-module of codimension $r$.
Let $0\to F_r\to\cdots\to F_0\to M\to0$ be a free resolution of $M$.
Applying $(-)^*$ to this gives an exact sequence $0\to F_0^*\to\cdots\to F_r^*\to\Ext_R^r(M,R)\to0$.
The module $\Ext_R^r(M,R)$ is CM of codimension $r$ by (2), and hence it is in $\perf^r(R)$.
Using (2) again, we observe that the functor $\Ext_R^r(-,R):\perf^r(R)\to\perf^r(R)$ is a duality.
\end{proof}

Now we give the definition of a perfect link of a CM module.

\begin{dfn}
Let $M$ be a CM $R$-module of codimension $r$.\\
(1) A homomorphism $f:P\to M$ is called a {\em perfect morphism} of $M$ if $f$ is surjective and $P$ is a perfect $R$-module of codimension $r$.\\
(2) For a perfect morphism $f$ of $M$, we define the {\em perfect link} $\li_fM$ of $M$ with respect to $f$ by $\li_fM=\Ext_R^r(\Ker f,\omega)$.
\end{dfn}

\begin{rem}\label{doko}
Every CM module admits a perfect morphism.
Indeed, let $M$ be a CM $R$-module of codimension $r$.
By Proposition \ref{cmrmcm}(2) we find an $R$-sequence $\xx=x_1,\dots,x_r$ such that $M$ is a MCM $R/\xx R$-module.
Take an epimorphism $f:P\to M$ with $P$ a free $R/\xx R$-module.
Then $f$ is a perfect morphism of $M$.
\end{rem}

Here are some fundamental properties of perfect links.

\begin{prop}\label{basli}
Let $M$ be a CM $R$-module of codimension $r$, and let $f:P\to M$ be a perfect morphism.
\begin{enumerate}[\rm(1)]
\item
The modules $\Ker f$ and $\li_fM$ are also CM $R$-modules of codimension $r$.
\item
An exact sequence $0\to\Ext^r(M,\omega)\xrightarrow{\Ext^r(f,\omega)}\Ext^r(P,\omega)\xrightarrow{g}\li_fM\to0$ is induced.
\item
If $R$ is Gorenstein, then $g$ is a perfect morphism and one has $\li_g\li_fM\cong M$.
\end{enumerate}
\end{prop}

\begin{proof}
(1) Note that there is an exact sequence
\begin{equation}\label{kpm}
0 \to \Ker f \to P \xrightarrow{f} M \to 0.
\end{equation}
By Proposition \ref{qres}(1) the $R$-module $\Ker f$ is CM of codimension $r$, and so is $\li_fM$ by Proposition \ref{qres}(2).

(2) Applying $\Ext^*(-,\omega)$ to \eqref{kpm} and using (1), we obtain an exact sequence as in the assertion.

(3) Proposition \ref{qres}(3) implies that $\Ext_R^r(P,R)$ is a perfect $R$-module of codimension $r$.
Hence $g$ is a perfect morphism of $\li_fM$, and we have $\li_g(\li_fM)=\Ext_R^r(\Ker g,R)\cong\Ext_R^r(\Ext_R^r(M,R),R)\cong M$, where the last isomorphism follows from Proposition \ref{qres}(2).
\end{proof}

The following result is an immediate consequence of Propositions \ref{basli}(1) and \ref{qres}(2).

\begin{cor}\label{1and2}
Let $M$ and $N$ be CM $R$-modules of codimension $r$.
Let $f:P\to M$ and $g:Q\to N$ be perfect morphisms.
Then $\li_fM\cong\li_gN$ if and only if $\Ker f\cong\Ker g$.
\end{cor}

According to Proposition \ref{basli}, taking perfect links preserves the CM property and the codimension.
We are thus interested in when given two CM modules of the same codimension are connected by taking perfect links repeatedly.
We make the following definitions.

\begin{dfn}
Let $M,N$ be CM $R$-modules of codimension $r$.\\
(1) If there exists a perfect morphism $f:P\to M$ such that $\li_fM$ is isomorphic to $N$, then we say that $M$ is {\em directly perfectly linked} to $N$ (by $P$), and denote this by $M\sim N$.\\
(2) We say that $M$ is {\em perfectly linked} to $N$ if there exist a finite number of CM $R$-modules $M_0,\dots,M_n$ of codimension $r$ such that $M=M_0\sim M_1\sim\cdots\sim M_n=N$.
When this is the case, we write $M\approx N$.\\
(3) If there exist modules $M_0,\dots,M_n$ as in (2) with $n=2$ (respectively, $n=3$), then $M$ is said to be {\em doubly perfectly linked} (respectively, {\em triply perfectly linked}) to $N$.
\end{dfn}

\begin{rem}
Assume that $R$ is Gorenstein.
It follows from Remark \ref{doko} and Proposition \ref{basli}(3) that each $M\in\cm^r(R)$ is doubly perfectly linked to $M$ itself, which especially says $M\approx M$.
It follows from Proposition \ref{basli}(3) again that if $M,N\in\cm^r(R)$ are such that $M\sim N$, then one has $N\sim M$.
Hence $\approx$ is an equivalence relation of the set of isomorphism classes of modules in $\cm^r(R)$.
\end{rem}

The direct perfect linkage is preserved by taking finite direct sums.

\begin{prop}
Let $M_1,\dots,M_n,N_1,\dots,N_n$ be CM $R$-modules of codimension $r$.
If $M_i\sim N_i$ for each $1\le i\le n$, then $\bigoplus_{i=1}^nM_i\sim\bigoplus_{i=1}^nN_i$.
\end{prop}

\begin{proof}
For each $1\le i\le n$ there exists an exact sequence $0 \to K_i \to P_i \xrightarrow{f_i} M_i \to 0$ with $f_i$ a perfect morphism such that $N_i$ is isomorphic to $\Ext_R^r(K_i,\omega)$.
Then there is an exact sequence $\textstyle0 \to \bigoplus_{i=1}^nK_i \to \bigoplus_{i=1}^nP_i \xrightarrow{f} \bigoplus_{i=1}^nM_i \to 0$, where $f:=\bigoplus_{i=1}^nf_i$ is a perfect morphism.
We thus have
$$
\textstyle\li_f(\bigoplus_{i=1}^nM_i)=\Ext_R^r(\bigoplus_{i=1}^nK_i,\omega)\cong\bigoplus_{i=1}^n\Ext_R^r(K_i,\omega)\cong\bigoplus_{i=1}^nN_i,
$$
which shows the proposition.
\end{proof}

A classical result on linkage of ideals asserts that all the complete intersection ideals of the same codimension belong to the same linkage class.
It would be interesting to compare with this fact the following proposition and corollary.

\begin{prop}\label{peli}
Let $R$ be Gorenstein.
Let $M,N$ be CM $R$-modules of codimension $r$.
\begin{enumerate}[\rm(1)]
\item
Assume that $M$ is perfectly linked to $N$.
Then $M$ is perfect if and only if so is $N$.
\item
Suppose that $M$ and $N$ are perfect.
Then $M$ is directly perfectly linked to $N$ by $M\oplus\Ext_R^r(N,R)$.
\end{enumerate}
\end{prop}

\begin{proof}
(1) Let $f:P\to M$ be a perfect morphism.
Taking the kernel, we have an exact sequence $0\to K \to P \xrightarrow{f} M \to 0$.
Suppose that $M$ is perfect.
Then Proposition \ref{qres}(1) implies that $K$ is also a perfect module of codimension $r$, and so is $\li_fM=\Ext_R^r(K,R)$ by Proposition \ref{qres}(3).
This shows the assertion.

(2) Put $P=M\oplus\Ext_R^r(N,R)$.
It is seen from Proposition \ref{qres}(3) that $P$ is a perfect module of codimension $r$.
Hence the projection $f:P\to M$ is a perfect morphism of $M$.
We have $\li_fM=\Ext_R^r(\Ker f,R)\cong\Ext_R^r(\Ext_R^r(N,R),R)\cong N$, where the last isomorphism follows from Proposition \ref{qres}(2).
\end{proof}

As a direct consequence of Proposition \ref{peli}(2), we have the following decisive result.

\begin{cor}
If $R$ is regular, then any two CM $R$-modules of the same codimension are directly perfectly linked.
\end{cor}

Thus, over a regular local ring there exists only one perfect linkage class for each codimension.
Consequently, our main interests are in perfect linkage of modules over singular local rings.

We close this section by stating a result on the minimal number of generators of a perfect module appearing in a perfect morphism linking two cyclic CM modules of the same codimension.
For a local ring $A$ with residue field $K$ and an $A$-module $M$ we denote by $\nu_A(M)$ the {\em minimal number of generators} of $M$, and by $r_A(M)$ the {\em type} of $M$, namely, $\nu_A(M)=\dim_K(M\otimes_AK)$ and $r_A(M)=\dim_K\Ext_A^{\depth M}(K,A)$.
Note that for an ideal $I$ of $A$ annihilating $M$ one has $\nu_A(M)=\nu_{A/I}(M)$ and $r_A(M)=r_{A/I}(M)$.

\begin{prop}
Let $I,J$ be ideals of $R$ with height $r$ such that $R/I$ is directly perfectly linked to $R/J$.
Then for any perfect morphism $P$ of $R/I$ one has $\nu_R(P)\le r_R(R/J)+1$.
In particular, if $R/J$ is Gorenstein, then $P$ is generated by at most two elements.
\end{prop}

\begin{proof}
First of all, note that $R/I$ and $R/J$ are CM $R$-modules of codimension $r$.
Let $0 \to K \to P \xrightarrow{f} R/I \to 0$ be an exact sequence of $R$-modules such that $f$ is a perfect morphism.
Then $K$ is isomorphic to $\Ext_R^r(R/J,\omega)$ by Proposition \ref{qres}(2).
The exact sequence gives an inequality $\nu_R(P)\le\nu_R(K)+\nu_R(R/I)=\nu_R(K)+1$, and we have $\nu_R(K)=\nu_R(\Ext_R^r(R/J,\omega))=r_R(R/J)$ by \cite[Proposition 3.3.11(b)]{BH}.
\end{proof}

\section{MCM approximations and Yoshino--Isogawa links}

In this section, we investigate the relationships of perfect links with syzygies, MCM approximations and Yoshino--Isogawa links, that is, links in the sense of Yoshino and Isogawa \cite{YI}.
First, we give a result on the structure of a MCM approximation of a perfect link, by which we recover the main result of Yoshino and Isogawa.
Next, we relate the condition that two CM modules $M,N$ of codimension $r$ are doubly perfectly linked with the existence of stable equivalences of syzygies, and in the case where $r\le1$ we obtain a complete criterion for $M,N$ to be doubly perfectly linked.

Recall that two $R$-modules $M,N$ are called {\em stably equivalent}, denoted by $M\scong N$, if there is an isomorphism $M\oplus F\cong N\oplus G$ of $R$-modules with $F,G$ free.
We denote by $\lcm(R)$ the {\em stable category} of $\mcm(R)$.
The objects of the category $\lcm(R)$ are defined to be the MCM $R$-modules, and the hom-set $\Hom_{\lcm(R)}(M,N)$ is defined as the quotient module of $\Hom_R(M,N)$ by the submodule consisting of homomorphisms $M\to N$ factoring through free $R$-modules.
Note that two MCM $R$-modules $M,N$ are stably equivalent if and only if they are isomorphic in the category $\lcm(R)$.
Dually, $\ucm(R)$ stands for the {\em costable category} of $\mcm(R)$.
The objects of $\ucm(R)$ are the MCM $R$-modules, and the hom-set $\Hom_{\ucm(R)}(M,N)$ is the quotient of $\Hom_R(M,N)$ by homomorphisms $M\to N$ factoring through finite direct sums of copies of the canonical module $\omega$.

For each $R$-module $M$, denote by $\syz M$ (or $\syz_RM$) and $\tr M$ (or $\tr_RM$) the {\em syzygy} and {\em transpose} of $M$ respectively, that is to say, if $F_1\xrightarrow{\phi}F_0\to M\to0$ is part of a minimal free resolution of $M$, then $\syz M$ is the image of $\phi$ and $\tr M$ is the cokernel of the $R$-dual of $\phi$.
Note that $\syz M$ and $\tr M$ are uniquely determined up to isomorphism because so is a minimal free resolution of $M$.

Suppose now that $R$ is Gorenstein.
Then $\lcm(R)$ and $\ucm(R)$ coincide, and $\syz,\tr$ induce an autoequivalence and a duality of $\lcm(R)$, respectively.
For each MCM $R$-module $M$ the {\em cosyzygy} $\syz^{-1}M$ of $M$ is defined as $(\syz(M^*))^*$.
This is uniquely determined up to isomorphism, and induces a quasi-inverse functor to the syzygy functor $\syz:\lcm(R)\to\lcm(R)$.
For an integer $n>0$ we denote by $\syz^n,\syz^{-n}$ the $n$-fold composition of $\syz,\syz^{-1}$ respectively.
For the details of (co)syzygies and transposes, we refer the reader to \cite{ABr}.

Next we recall the definition of a MCM approximation introduced by Auslander and Buchweitz \cite{ABu}.

\begin{dfn}
Let $M$ be an $R$-module.
A homomorphism $f:X\to M$ of $R$-modules with $X$ MCM is called a {\em MCM approximation} of $M$ if every homomorphism $f':X'\to M$ of $R$-modules with $X'$ MCM factors through $f$.
When this is the case, we set $\x_RM=X$.
\end{dfn}

\begin{rem}\label{mcmappr}
If $f:X\to M$ is a MCM approximation, then one has an exact sequence
$$
0 \to Y \to X \xrightarrow{f} M \to 0
$$
of $R$-modules such that $Y$ has finite injective dimension by \cite[Theorem A]{ABu}.
Conversely, if there exists such an exact sequence, then $f$ is a MCM approximation by \cite[Theorem 2.3]{ABu}.
If $f_1:X_1\to M$ and $f_2:X_2\to M$ are MCM approximations, then $X_1\oplus\omega^{\oplus n_1}\cong X_2\oplus\omega^{\oplus n_2}$ as $R$-modules for some $n_1,n_2\ge0$; see \cite[Theorem B]{ABu}.
In other words, $X_1\cong X_2$ in the costable category $\ucm(R)$.
In particular, when $R$ is Gorenstein, $\x_RM$ is uniquely determined up to stable equivalence for each $R$-module $M$, and in fact one has $\x_RM\scong\syz^{-i}\syz^iM$ for any $i\ge d-\depth M$ by \cite[Proposition (2.21) and Corollary (4.22)]{ABr}.
\end{rem}

Now we state and prove the first main result of this section, which clarifies the structure of a MCM approximation of a perfect link.

\begin{thm}\label{mg}
Let $M$ be a CM $R$-module of codimension $r$, and let $f:P\to M$ be a perfect morphism.
Then there exist exact sequences of $R$-modules
$$
0 \to Z \to \li_g(\syz^rM) \xrightarrow{\phi} \li_fM \to 0,\qquad
0 \to \omega^{\oplus n_{r-1}} \xrightarrow{\partial_{r-1}} \cdots \xrightarrow{\partial_1} \omega^{\oplus n_0} \to Z \to 0,
$$
where $g$ is a perfect morphism of $\syz^rM$ induced by $\syz^rf$, and $\partial_j$ is represented by an $n_{j-1}\times n_j$ matrix with entries in $\m$ for $1\le j\le r-1$.
In particular, $\phi$ is a MCM approximation of $\li_fM$, and one has an isomorphism $\x_R(\li_fM)\cong\li_g(\syz^rM)$ in $\ucm(R)$.
\end{thm}

\begin{proof}
Taking the kernel of $f$, we have an exact sequence $0 \to K \to P \xrightarrow{f} M \to 0$.
According to Proposition \ref{basli}(1), all the modules appearing in this sequence are CM of codimension $r$.
Taking the $r$th syzygies gives rise to an exact sequence
$$
0 \to \syz^rK \to \syz^rP\oplus F \xrightarrow{g} \syz^rM \to 0
$$
of $R$-modules, where $F$ is free and $g$ is of the form $(\syz^rf, h)$ with $h\in\Hom_R(F,\syz^rM)$.
The depth lemma implies that this is an exact sequence of MCM $R$-modules (hence $\syz^rP$ is free), and $g$ is a perfect morphism of $\syz^rM$.
We thus have $\li_g(\syz^rM)=\Hom_R(\syz^rK,\omega)$.
Take a minimal free resolution
$$
\cdots \to R^{\oplus m_j} \xrightarrow{A_j} R^{\oplus m_{j-1}} \to \cdots \to R^{\oplus m_1} \xrightarrow{A_1} R^{\oplus m_0} \to K \to 0
$$
of $K$ over $R$, where each $A_j$ is a matrix with entries in $\m$.
As $f$ is a perfect morphism and $K$ is a CM module of codimension $r$, the module $\Ext_R^j(K,\omega)$ vanishes for $j<r$ and is isomorphic to $\li_fM$ for $j=r$.
Hence, dualizing the above resolution by $\omega$ yields an exact sequence
$$
0 \to \omega^{\oplus m_0} \xrightarrow{{}^t\!A_1} \omega^{\oplus m_1} \to \cdots \to \omega^{\oplus m_{r-2}} \xrightarrow{{}^t\!A_{i-1}} \omega^{\oplus m_{r-1}} \xrightarrow{\psi} \li_g(\syz^rM) \xrightarrow{\phi} \li_fM \to 0.
$$
Letting $Z$ be the image of $\psi$, we obtain exact sequences as in the theorem.
The last assertion of the theorem follows from Remark \ref{mcmappr}.
\end{proof}

We recall the definition of the Yoshino--Isogawa link of a MCM module \cite{YI}.

\begin{dfn}[Yoshino--Isogawa]
Let $R$ be a Gorenstein local ring.
Then the assignment $M\mapsto\Hom_R(\syz M,R)$ makes an additive contravariant functor
$$
\link_R:\lcm(R)\to\lcm(R),
$$
which is in fact a duality.
For each MCM $R$-module, we call $\link_RM$ the {\em Yoshino--Isogawa link} of $M$.
\end{dfn}

Applying our Theorem \ref{mg} to a Gorenstein local ring, we obtain the following result on the relationships between MCM approximations, Yoshino--Isogawa links, syzygies, transposes and our perfect links.

\begin{cor}\label{genyi}
Let $R$ be a Gorenstein local ring.
Let $M$ be a CM $R$-module of codimension $r$ with a perfect morphism $f:P\to M$.
Then one has stable equivalences
$$
\x_R\li_fM\scong\link_R\syz^rM,\qquad\syz^r\li_fM\scong\syz^{r+1}\tr\syz^rM.
$$
\end{cor}

\begin{proof}
Using Theorem \ref{mg}, we have an exact sequence $0 \to \syz^rK \to \syz^rP\oplus F\xrightarrow{g}\syz^rM\to0$ with $F$ free, and an isomorphism $\x_R(\li_fM)\cong\li_g(\syz^rM)$ in $\ucm(R)=\lcm(R)$.
As $\syz^rP$ is a free $R$-module, we have $\syz^rK\scong\syz^{r+1}M$.
Hence it holds that $\li_g(\syz^rM)=\Hom_R(\syz^rK,R)\scong\Hom_R(\syz^{r+1}M,R)=\link_R(\syz^rM)$, which gives the first isomorphism.
Since $R$ is Gorenstein, $\x_RN$ is stably equivalent to $\syz^{-r}\syz^rN$ for each $N\in\cm^r(R)$ by Remark \ref{mcmappr}.
Taking the $r$th syzygy of the first isomorphism gives the second one.
\end{proof}

We immediately deduce the following corollary, which is the main result of \cite{YI}.

\begin{cor}[Yoshino--Isogawa]\label{coryi}
Let $R$ be a Gorenstein local ring, and let $\xx=x_1,\dots,x_r$ be an $R$-sequence.
Then there is a diagram of functors
$$
\begin{CD}
\lcm(R/\xx R) @>{\syz_R^r}>> \lcm(R) \\
@VV{\link_{R/\xx R}}V @VV{\link_R}V \\
\lcm(R/\xx R) @>{\x_R}>> \lcm(R)
\end{CD}
$$
which is commutative up to isomorphism.
\end{cor}

\begin{proof}
Let $M\in\lcm(R/\xx R)$.
Proposition \ref{cmrmcm}(1) implies that $M$ is a CM $R$-module of codimension $r$.
Let $f$ be a perfect morphism as in Remark \ref{doko}.
The assertion follows from Corollary \ref{genyi}.
\end{proof}

Next we investigate perfect morphisms whose kernels are isomorphic.

\begin{prop}\label{2and3}
Let $M$ and $N$ be CM $R$-modules of codimension $r$.
\begin{enumerate}[\rm(1)]
\item
Let $f:P\to M$ and $g:Q\to N$ be perfect morphisms such that $\Ker f\cong\Ker g$.
Then $\syz^{r+1}M$ is stably equivalent to $\syz^{r+1}N$.
\item
Suppose $r\le1$.
If $\syz^rM$ is stably equivalent to $\syz^rN$, then there exist perfect morphisms $f,g$ of $M,N$ respectively, such that $\Ker f\cong\Ker g$.
\end{enumerate}
\end{prop}

\begin{proof}
(1) There are exact sequences $0\to\Ker f\to P\xrightarrow{f}M\to0$ and $0\to\Ker g\to Q\xrightarrow{g}N\to0$.
Note that $\syz^rP$ and $\syz^rQ$ are free $R$-modules.
Applying $\syz^r$ to these exact sequences yields $\syz^{r+1}M\scong\syz^r\Ker f\cong\syz^r\Ker g\scong\syz^{r+1}N$.

(2) When $r=0$, the $R$-modules $M,N$ are MCM and we have an isomorphism $M\oplus R^a\cong N\oplus R^b$ for some $a,b\ge0$.
There are exact sequences $0\to\syz M\to R^m\xrightarrow{f}M\to0$ and $0\to\syz N\to R^n\xrightarrow{g}N\to 0$, and we have $\Ker f=\syz M=\syz(M\oplus R^a)\cong\syz(N\oplus R^b)=\syz N=\Ker g$.

Now let $r=1$.
Take exact sequences $0\to\syz M\xrightarrow{\alpha}R^m\to M\to0$ and $0\to\syz N\xrightarrow{\beta}R^n\to N\to0$.
As $M$ has positive grade, $\syz M$ has rank $m$.
Hence there exists an exact sequence $0\to R^m\xrightarrow{\gamma}\syz M\xrightarrow{\delta}C\to0$ such that $C$ has rank $0$.
Since $\syz M$ is stably equivalent to $\syz N$, we have an isomorphism $X:=\syz M\oplus R^a\cong\syz N\oplus R^b$ for some $a,b\ge0$.
The pushout diagram of $\alpha$ and $\delta$ gives rise to exact sequences
$$
0\to R^m\to R^m\to P\to0,\qquad0\to C\to P\xrightarrow{f}M\to0.
$$
Adding the identity maps of $R^a,R^b$ to $\gamma,\beta$ respectively makes exact sequences $0\to R^{m+a}\to X\xrightarrow{\varepsilon}C\to0$ and $0\to X\xrightarrow{\zeta}R^{n+b}\to N\to0$.
From the pushout diagram of $\varepsilon$ and $\zeta$ we get exact sequences
$$
0\to R^{m+a}\to R^{n+b}\to Q\to0,\qquad0\to C\to Q\xrightarrow{g}N\to0.
$$
As $C$ has rank $0$, it has positive grade, whence $\dim C\le d-1$.
Since $R^m$ and $\syz M$ are MCM modules, it is observed that $\depth C\ge d-1$.
Therefore $C$ is a CM module of codimension $1$.
Proposition \ref{qres}(1) implies that $P,Q$ are perfect $R$-modules of codimension $1$.
Thus $f,g$ are perfect morphisms of $M,N$ respectively, and we have $\Ker f\cong C\cong\Ker g$.
\end{proof}

\begin{ques}
Does Proposition \ref{2and3}(2) hold even when $r\ge2$\,?
\end{ques}

Finally we obtain the following result on double perfect linkage, which is the second main result of this section.

\begin{thm}\label{main}
Let $R$ be a Gorenstein local ring.
Let $M,N$ be CM $R$-modules of codimension $r$.
Let $f,g$ be perfect morphisms of $M,N$ respectively.
Consider the following four conditions:
\begin{center}
{\rm(1)} $\li_fM\cong\li_gN$,\qquad
{\rm(2)} $\Ker f\cong\Ker g$,\qquad
{\rm(3)} $\syz^rM\scong\syz^rN$,\qquad
{\rm(4)} $\syz^r\li_fM\scong\syz^r\li_gN$.
\end{center}
Then the implications ${\rm(1)}\Leftrightarrow{\rm(2)}\Rightarrow{\rm(3)}\Leftrightarrow{\rm(4)}$ hold.
If $r\le1$ and {\rm(3)} is satisfied, then {\rm(2)} is also satisfied for some perfect morphisms $f,g$ of $M,N$ respectively.
\end{thm}

\begin{proof}
The conditions (1) and (2) are equivalent by Corollary \ref{1and2}.
It is clear that (1) implies (4).
Using Corollary \ref{genyi} and the fact that $\syz^{r+1}\tr:\lcm(R)\to\lcm(R)$ is a duality, we see that (3) is equivalent to (4).
The first assertion now follows.
The second assertion is shown by Proposition \ref{2and3}(2).
\end{proof}

\section{Applications}

In this section, first we apply our results obtained in the preceding sections to classical CI-linkage to investigate when two Cohen--Macaulay ideals $I,J$ of codimention $r$ to be CI-linked.
Next we prove that in codimension one every perfect link is either a double one or a triple one, and establish criteria for two CM modules to be doubly and triply perfectly linked.
We also construct a lot of examples of perfect linkage of modules and CI-linkage of ideals.

Let $I$ be an ideal of $R$.
The {\em codimension} $\codim I$ of $I$ is by definition the codimension of the $R$-module $R/I$, i.e., $\codim I=\dim R-\dim R/I$.
We say that $I$ is a {\em CM} ideal if the residue ring $R/I$ is a CM ring.
Hence $I$ is a CM ideal of codimension $r$ if and only if $R/I$ belongs to $\cm^r(R)$.

Let $I,J$ be proper ideals of $R$.
We say that $I$ is {\em directly CI-linked} to $J$ if there exists an $R$-sequence $\xx=x_1,\dots,x_r$ in $I\cap J$ such that $I=(\xx R:J)$ and $J=(\xx R:I)$.
Note that in this case $I,J$ have codimension $r$.
Note also that if $R$ is Gorenstein, then the equality $I=(\xx R:J)$ implies the equality $J=(\xx R:I)$ and vice versa; see \cite[Exercise 3.2.15]{BH}.
We say that $I$ is {\em doubly CI-linked} to $J$ if there is an ideal $K$ of $R$ to which $I,J$ are directly CI-linked.
Also, $I$ is said to be {\em CI-linked} to $J$ if there is a sequence of ideals $I_0,\dots,I_n$ with $I_0=I$ and $I_n=J$ such that $I_i$ is directly CI-linked to $I_{i+1}$ for all $0\le i\le n-1$.

We start by stating the proposition below, which guarantees that the classical notion of CI-linkage of CM ideals is a special case of our notion of perfect linkage of CM modules over perfect modules.

\begin{prop}\label{liid}
Let $R$ be a Gorenstein local ring.
\begin{enumerate}[\rm(1)]
\item
Let $I$ be a CM ideal of $R$ with codimension $r$.
Let $\xx=x_1,\dots,x_r$ be an $R$-sequence in $I$.
Then for the natural surjection $f:R/\xx R\to R/I$ one has $\Ker f=I/\xx R$ and $\li_f(R/I)\cong R/(\xx R:I)$.
\item
Let $I,J$ be CM ideals of $R$ with codimension $r$.
If $I$ is CI-linked (respectively, directly CI-linked, doubly CI-linked) to $J$, then $R/I$ is perfectly linked (respectively, directly perfectly linked, doubly perfectly linked) to $R/J$.
\end{enumerate}
\end{prop}

\begin{proof}
(1) It is trivial that the kernel of $f$ is $I/\xx R$, and we get an exact sequence $0\to I/\xx R\to R/\xx R\xrightarrow{f} R/I\to0$.
It follows that $\li_f(R/I)=\Ext_R^r(I/\xx R,R)\cong\Hom_{R/\xx R}(I/\xx R,R/\xx R)$.
Note that $R/I$ is a MCM module over the Gorenstein local ring $R/\xx R$.
Applying $\Hom_{R/\xx R}(-,R/\xx R)$ induces an exact sequence $0\to\Hom_{R/\xx R}(R/I,R/\xx R)\to R/\xx R\to\li_f(R/I)\to0$.
Since $\Hom_{R/\xx R}(R/I,R/\xx R)\cong(0:_{R/\xx R}I/\xx R)=(\xx R:_RI)/\xx R$, we obtain $\li_f(R/I)\cong R/(\xx R:_RI)$.

(2) The assertion immediately follows from (1).
\end{proof}

Using the above Proposition \ref{liid}, one deduces from Theorem \ref{main} the following two propositions.

\begin{prop}\label{ixjy}
Suppose that $R$ is a Gorenstein local ring.
Let $I,J$ be CM ideals of $R$ with codimension $r$.
The following are equivalent.
\begin{enumerate}[\rm(1)]
\item
The ideals $I$ and $J$ are doubly CI-linked.
\item
There exist $R$-sequences $\xx=x_1,\dots,x_r$ in $I$ and $\yy=y_1,\dots,y_r$ in $J$ such that $I/\xx R\cong J/\yy R$ as $R$-modules.
\end{enumerate}
\end{prop}

\begin{proof}
Let $\xx=x_1,\dots,x_r$ and $\yy=y_1,\dots,y_r$ be $R$-sequences in $I$ and $J$, respectively.
Put $K=(\xx R:I)$ and $L=(\yy R:J)$.
Since $R$ is Gorenstein, $I$ and $J$ are directly CI-linked to $K$ and $L$, respectively.
Let $f:R/\xx R\to R/I$ and $g:R/\yy R\to R/J$ be the natural surjections.
By Proposition \ref{liid}(1) we have
$$
\Ker f=I/\xx R,\quad \Ker g=J/\yy R,\quad \li_f(R/I)\cong R/K,\quad \li_g(R/J)\cong R/L.
$$

Assume $K=L$.
Then Theorem \ref{main} says that $I/\xx R$ is isomorphic to $J/\yy R$.
This shows that (1) implies (2).
Conversely, suppose $I/\xx R\cong J/\yy R$.
It then follows from Theorem \ref{main} that $R/K$ is isomorphic as an $R$-module to $R/L$.
Taking the annihilators over $R$ implies $K=L$.
This shows that (2) implies (1).
\end{proof}

\begin{prop}\label{ic1}
Suppose that $R$ is Gorenstein.
Let $I,J$ be CM ideals of $R$ with codimension one.
The following are equivalent.
\begin{enumerate}[\rm(1)]
\item
The ideal $I$ is doubly CI-linked to $J$.
\item
The $R$-module $I/xR$ is isomorphic to $J/yR$ for some nonzerodivisors $x\in I$ and $y\in J$.
\item
The $R$-module $I$ is stably equivalent to $J$.
\item
The $R$-module $(xR:I)$ is stably equivalent to $(yR:J)$ for some nonzerodivisors $x\in I$ and $y\in J$.
\end{enumerate}
\end{prop}

\begin{proof}
(1) $\Leftrightarrow$ (2):
The equivalence follows from Proposition \ref{ixjy}.

(2) $\Rightarrow$ (3) $\Leftrightarrow$ (4):
Let $x\in I$ and $y\in J$ be nonzerodivisors of $R$, and let $f:R/xR\to R/I$ and $g:R/yR\to R/J$ be the natural surjections.
Then Proposition \ref{liid}(1) implies $\Ker f=I/xR$, $\Ker g=J/yR$, $\li_f(R/I)\cong R/(xR:I)$ and $\li_g(R/J)\cong R/(xR:J)$.
Theorem \ref{main} deduces the implications
$$
I/xR\cong J/yR\quad\Rightarrow\quad I\scong J\quad\Leftrightarrow\quad(xR:I)\scong(yR:J).
$$

(3) $\Rightarrow$ (2):
The implication is essentially shown in the proof of the corresponding implication in Theorem \ref{main}, or more precisely, in the proof of Proposition \ref{2and3}(2).
We can give a more explicit proof in our current settings, so let us do it.
Suppose that $I$ is stably equivalent to $J$.

We claim that $I$ is isomorphic to $J$ as an $R$-module.
In fact, one has $I\oplus R^{\oplus a}\cong J\oplus R^{\oplus b}$ for some integers $a,b\ge0$.
As $I,J$ have rank one, we have $a=b$.
If $I$ is a free $R$-module, then so is $J$, and $I\cong J\cong R$.
So let us assume that $I$ and $J$ are nonfree $R$-modules.
Denoting by $\widehat{(-)}$ the $\m$-adic completion, we have an isomorphism $\widehat{I}\oplus\widehat{R}^{\oplus a}\cong\widehat{J}\oplus\widehat{R}^{\oplus a}$.
Since $\widehat I$ and $\widehat J$ are nonfree $\widehat R$-modules, the Krull--Schmidt theorem implies $\widehat I\cong\widehat J$, and hence $I\cong J$ by \cite[Exercise 7.5]{E}.
Now the claim follows.

Let $\phi:I\to J$ be an isomorphism of $R$-modules.
As $I$ has positive grade, there exists an $R$-regular element $x\in I$.
Putting $y=\phi(x)$, we see that $y$ is an $R$-regular element in $J$.
The isomorphism $\phi$ induces an isomorphism $I/xR\to J/yR$.
Thus the statement (2) holds.
\end{proof}

In what follows, we concentrate on investigate perfect linkage in codimension one by taking advantage of Theorem \ref{main} together with the above Proposition \ref{ic1}.
First of all, we construct an example of CI-linkage of ideals of a singular local ring.
See also Example \ref{xy} stated later.

\begin{ex}
Let $R=k[[x,y]]/(xy)$, where $k$ is a field.
Let $I$ be a nonfree $\m$-primary ideal, where $\m$ is the maximal ideal of $R$.
Then $I$ is directly CI-linked to $\m$.
In particular, one has $R/I\sim k$.
\end{ex}

\begin{proof}
The second assertion follows from Proposition \ref{liid}(2).
We show the first assertion.
It follows from \cite[Theorem 1.2]{S} that $\nu_R(I)\le2$.
As $I$ is not free, we have $\nu_R(I)=2$.
We establish two claims.

\begin{claim}\label{c1}
One has $I=(x^m,y^n)$ for some $m,n\in\Z_{>0}$.
\end{claim}

\begin{proof}[Proof of Claim \ref{c1}]
Write $I=(f,g)$ for some $f,g\in R$.
There exist $p,q,r,s\in k$ and $a,b,c,d\in\Z_{>0}$ such that $f=px^a+qy^b$ and $g=rx^c+sy^d$.

Let us first consider the case $p=0$.
In this case we may assume $q=r=1$, so $f=y^b$ and $g=x^c+sy^d$.
If $b\le d$, then we may assume $s=0$, and we are done.
Let $b>d$, and assume $s\ne0$.
Then $f=y^b\equiv(-s^{-1}x^c)y^{b-d}=0$ mod $g$, which is a contradiction.
Hence $s=0$, and we are done.

The case $r=0$ is similarly handled, so now let us consider the case $p,r\ne0$.
Then we may assume $p=r=1$ and $a\le c$.
When $a<c$, we have $g=x^c+sy^d\equiv(-qy^b)x^{c-a}+sy^d=sy^d$ mod $f$, and we can reduce to the case $r=0$.
When $a=c$, we have $g\equiv-qy^b+sy^d=ty^e$ mod $f$ for some $t\in k$ and $e\in\Z_{>0}$, and reduce to the case $r=0$.
\renewcommand{\qedsymbol}{$\square$}
\end{proof}

\begin{claim}\label{c2}
One has $I=((x^m-y^n):\m)$ for some $m,n\in\Z_{>0}$.
\end{claim}

\begin{proof}[Proof of Claim \ref{c2}]
By Claim \ref{c1} we have $xI=(x^{m+1})=(x(x^m-y^n))\subseteq(x^m-y^n)$, and similarly $yI\subseteq(x^m-y^n)$.
Hence $I$ is contained in $((x^m-y^n):\m)$.
The ring $R/I$ has the $k$-basis $1,x,\dots,x^{m-1},y,\dots,y^{n-1}$ and so $\ell_R(R/I)=m+n-1$.
The ring $R/((x^m-y^n):\m)$ is the quotient of the artinian Gorenstein local ring $R/(x^m-y^n)$ by its socle.
The ring $R/(x^m-y^n)$ has the $k$-basis $1,x,\dots,x^{m-1},y,\dots,y^n$ if $m\le n$.
Hence $\ell_R(R/((x^m-y^n):\m))=\ell_R(R/(x^m-y^n))-1=m+n-1$.
Consequently, we have an equality $I=((x^m-y^n):\m)$.
\renewcommand{\qedsymbol}{$\square$}
\end{proof}
Claim \ref{c2} shows that $I$ is directly CI-linked to $\m$ by the nonzerodivisor $x^m-y^n$ of $R$.
\end{proof}

Next we construct an example of perfect linkage of CM modules, and compare several numerical invariants for modules.
It turns out that taking a perfect link do not preserve any of the minimal number of generators, type and length.

\begin{ex}
Let $R=k[[x,y]]/(x^2-y^3)$, where $k$ is a field.
Denote by $\m$ the maximal ideal of $R$.
\begin{enumerate}[(1)]
\item
One has that $\m$ is directly CI-linked to $(x,y^2)$, and $(x,y^2)$ is directly CI-linked to $(x^2,xy)$.
\item
One has $k\sim R/(x,y^2)\sim R/(x^2,xy)$ and $\m/(x^2)\sim k\sim\m/(y^2)$.
In particular, the CM $R$-modules $k,R/(x,y^2),R/(x^2,xy),\m/(x^2),\m/(y^2)$ of codimension $1$ are perfectly linked with one another.
\item
The following is a list of minimal numbers of generators, types and lengths.
\begin{center}
{\tabcolsep=5mm
\begin{tabular}{|c|c|c|c|c|c|}
\hline
& $k$ & $R/(x,y^2)$ & $R/(x^2,xy)$ & $\m/(x^2)$ & $\m/(y^2)$\\
\hline
$\nu_R$ & $1$ & $1$ & $1$ & $2$ & $2$\\
\hline
$r_R$ & $1$ & $1$ & $2$ & $1$ & $1$\\
\hline
$\ell_R$ & $1$ & $2$ & $4$ & $5$ & $3$\\
\hline
\end{tabular}
}
\end{center}
\end{enumerate}
In particular, having type $1$ is not preserved by double CI-linkage of ideals.
Compare this with the fact that having type $1$ is preserved by double CI-linkage of ideals in codimension at most $3$ for a regular ring; see \cite{PS,W}.
\end{ex}

\begin{proof}
(1) It is easy to see that $((x):\m)=(x,y^2)=((x^2):(x^2,xy))$.
Since $x$ and $x^2$ are nonzerodivisors of $R$, the ideal $\m$ is directly CI-linked to $(x,y^2)$, which is directly CI-linked to $(x^2,xy)$.

(2) It follows from (1) and Proposition \ref{liid}(2) that $k\sim R/(x,y^2)\sim R/(x^2,xy)$.
We have an exact sequence $0\to\syz^2k\xrightarrow{\theta}R^{\oplus2}\xrightarrow{\left(\begin{smallmatrix}x&y\end{smallmatrix}\right)}R\to k\to0$.
Since $\det\left(\begin{smallmatrix}
x&y\\
0&x
\end{smallmatrix}\right)=x^2$ is a nonzerodivisor of $R$, there is an exact sequence $0 \to R^{\oplus2} \xrightarrow{\left(\begin{smallmatrix}
x&y\\
0&x
\end{smallmatrix}\right)} R^{\oplus2} \to P \to 0$.
The $R$-module $P$ has rank $0$, so it belongs to $\perf^1(R)$.
There is a commutative diagram
$$
\begin{CD}
0 @>>> 0 @>>> R^{\oplus2} @= R^{\oplus2} @>>> 0 \\
@. @VVV @V{\left(\begin{smallmatrix}
x&y\\
0&x
\end{smallmatrix}\right)}VV @V{\left(\begin{smallmatrix}
x&y
\end{smallmatrix}\right)}VV \\
0 @>>> R @>{\left(\begin{smallmatrix}
0\\
1
\end{smallmatrix}\right)}>> R^{\oplus2} @>{\left(\begin{smallmatrix}
1&0
\end{smallmatrix}\right)}>> R @>>> 0
\end{CD}
$$
with exact rows, and applying the snake lemma to this gives rise to an exact sequence
\begin{equation}\label{4t}
0 \to \syz^2k \xrightarrow{\delta} R \to P \xrightarrow{f} k \to 0,
\end{equation}
where $f$ is a perfect morphism of $k\in\cm^1(R)$.
Note that when viewing $\syz^2k$ as a submodule of $R^{\oplus2}$ via the injection $\theta$, the map $\delta$ sends $\binom{a}{b}\in\syz^2k$ to $bx\in R$.
The minimal free resolution of $k$ is
$$
\cdots\xrightarrow{\left(\begin{smallmatrix}
-y^2&-x\\
x&y
\end{smallmatrix}\right)} R^{\oplus2} \xrightarrow{\left(\begin{smallmatrix}
y&x\\
-x&-y^2
\end{smallmatrix}\right)} R^{\oplus2} \xrightarrow{\left(\begin{smallmatrix}
-y^2&-x\\
x&y
\end{smallmatrix}\right)} R^{\oplus2} \xrightarrow{\left(\begin{smallmatrix}
y&x\\
-x&-y^2
\end{smallmatrix}\right)} R^{\oplus 2} \xrightarrow{\left(\begin{smallmatrix}
x&y
\end{smallmatrix}\right)}R\to k\to0,
$$
and there is a commutative diagram
$$
\begin{CD}
R^{\oplus2} @>{\left(\begin{smallmatrix}
y&x\\
-x&-y^2
\end{smallmatrix}\right)}>> R^{\oplus2} @>>> \m @>>> 0 \\
@V{\left(\begin{smallmatrix}
0&1\\
1&0
\end{smallmatrix}\right)}V{\cong}V @V{\left(\begin{smallmatrix}
0&1\\
1&0
\end{smallmatrix}\right)}V{\cong}V @V{\phi}V{\cong}V \\
R^{\oplus2} @>{\left(\begin{smallmatrix}
-y^2&-x\\
x&y
\end{smallmatrix}\right)}>> R^{\oplus2} @>>> \syz^2k @>>> 0
\end{CD}
$$
with exact rows and isomorphic vertical maps.
It is seen that the composite map $\delta\phi:\m\to R$ is a multiplication map by $-y^2$.
From \eqref{4t} we obtain a short exact sequence $0 \to R/(xy^2,x^2) \to P \xrightarrow{f} k \to 0$.
Thus we have
$$
\li_fk=\Ext_R^1(R/(xy^2,x^2),R)\cong\Hom_{R/(x^2)}(R/(xy^2,x^2),R/(x^2))\cong(0:_{R/(x^2)}xy^2)=\m/(x^2),
$$
which shows $\m/(x^2)\sim k$.
Similarly, there are an exact sequence $0 \to R^{\oplus2} \xrightarrow{\left(\begin{smallmatrix}
y&x\\
0&y
\end{smallmatrix}\right)} R^{\oplus2} \to Q \to 0$ and a commutative diagram
$$
\begin{CD}
0 @>>> 0 @>>> R^{\oplus2} @= R^{\oplus2} @>>> 0\phantom{.} \\
@. @VVV @V{\left(\begin{smallmatrix}
y&x\\
0&y
\end{smallmatrix}\right)}VV @V{\left(\begin{smallmatrix}
y&x
\end{smallmatrix}\right)}VV \\
0 @>>> R @>{\left(\begin{smallmatrix}
0\\
1
\end{smallmatrix}\right)}>> R^{\oplus2} @>{\left(\begin{smallmatrix}
1&0
\end{smallmatrix}\right)}>> R @>>> 0.
\end{CD}
$$
From these we get an exact sequence $0 \to R/(xy,y^2) \to Q \xrightarrow{g} k \to 0$, and an isomorphism $\li_gk\cong\m/(y^2)$.
Thus $\m/(y^2)\sim k$.

(3) The list is obtained by direct calculations.
For instance, the fact that $\m/(x^2)$ has type $1$ follows from the isomorphisms $\Hom_R(k,\m/(x^2))\cong(0:_{\m/(x^2)}\m)=(0:_{R/(x^2)}\m)=\soc R/(x^2)\cong k$.
\end{proof}

In the case where there is only one nonfree indecomposable MCM module, perfect linkage classes are rather restrictive as follows.

\begin{prop}\label{wa}
Let $R$ be a Gorenstein local ring possessing exactly one isomorphism class of a nonfree indecomposable MCM $R$-module.
(For example, the rings $\C[[x,y]]/(x^2-y^3)$ and $\C[[x,y,z]]/(x^2-yz)$ satisfy this condition.)
\begin{enumerate}[\rm(1)]
\item
Let $M,N$ be CM $R$-modules of codimension $1$ and infinite projective dimension.
Then there exist integers $m,n\ge1$ such that $M^{\oplus m}\sim N^{\oplus n}$.
\item
Let $I,J$ be CM ideals $I,J$ of $R$ with codimension $1$ and infinite projective dimension.
Then $I$ is isomorphic to $J$ as an $R$-module, and hence $I$ is doubly CI-linked to $J$.
\end{enumerate}
\end{prop}

\begin{proof}
Let $C$ be a nonfree indecomposable MCM $R$-module.
By assumption every nonfree indecomposable MCM $R$-module is isomorphic to $C$.

(1) Since $\syz M$ is a nonfree MCM $R$-module, it is stably equivalent to $C^{\oplus n}$ for some $n\ge1$.
Set $E:=\Ext_R^1(N,R)$, and let $0 \to Y \to X \xrightarrow{f} E \to 0$ be an exact sequence such that $f$ is a MCM approximation (hence $Y$ has finite projective dimension).
Then $X$ is stably equivalent to $C^{\oplus m}$ for some $m\ge0$.
If $m=0$, then $E$ has finite projective dimension, and so does $N$ by Proposision \ref{qres}, which is a contradiction.
Hence $m\ge1$.
There are stable equivalences $X^{\oplus n}\scong C^{\oplus mn}\scong\syz M^{\oplus m}$, so
$$
X^{\oplus n}\oplus R^{\oplus a}\cong\syz M^{\oplus m}\oplus R^{\oplus b}=:T
$$
for some $a,b\ge0$.
Since $E$ is CM of codimension $1$ and $X$ is MCM, the depth lemma implies that $Y$ is a free $R$-module.
We have two exact sequences
$$
0 \to Y^{\oplus n}\oplus R^{\oplus a} \xrightarrow{g} T \to E^{\oplus n} \to 0,\qquad
0 \to T \xrightarrow{h} R^{\oplus c} \to M^{\oplus m} \to 0
$$
for some $c$.
A pushout diagram gives rise to exact sequences
$$
0 \to Y^{\oplus n}\oplus R^{\oplus a} \xrightarrow{hg} R^{\oplus c} \to P \to 0,\qquad
0 \to E^{\oplus n}\to P\xrightarrow{\phi} M^{\oplus m} \to0.
$$
Note that $m,n$ are positive.
It is observed from these exact sequences that $P$ is a perfect $R$-module of codimension $1$, whence $\phi$ is a perfect morphism.
Proposition \ref{qres} yields $\li_\phi(M^{\oplus m})=\Ext_R^1(E^{\oplus n},R)\cong N^{\oplus n}$.

(2) The ideals $I$ and $J$ are nonfree MCM $R$-modules, and indecomposable as $R$ is a domain.
Hence $I\cong C\cong J$.
Proposition \ref{ic1} shows that $I$ and $J$ are doubly CI-linked.
\end{proof}

In the following example, there exist more than one isomorphism class of a nonfree indecomposable MCM module, so the above Proposition \ref{wa} cannot be applied.
However, one can obtain the same conclusion even in this case.

\begin{ex}\label{xy}
Let $k$ be an algebraically closed field of characteristic $0$.
Let $R=k[[x,y]]/(xy)$.
Denote by $\m$ the maximal ideal of $R$.
\begin{enumerate}[(1)]
\item
The first syzygy of an $R$-module of finite length is isomorphic to a finite direct sum of copies of $R$ and $\m$.
\item
Let $M$ and $N$ be CM $R$-modules of codimension one and with infinite projective dimension.
Then $M^{\oplus m}\sim N^{\oplus n}$ for some $m,n\ge1$.
\end{enumerate}
\end{ex}

\begin{proof}
(1) The assertion follows from \cite[Example 1.3]{K}, which is stated without a proof.
We show the assertion for the convenience of the reader.
By \cite[(9.9)]{Y}, the nonisomorphic indecomposable MCM $R$-modules are $R$, $R/xR$ and $R/yR$.
Let $M$ be an $R$-module of finite length.
Then $\syz M$ is MCM, so one can write $\syz M\cong R^{\oplus a}\oplus(R/xR)^{\oplus b}\oplus(R/yR)^{\oplus c}$ for some $a,b,c\ge0$.
There is an exact sequence $0\to\syz M\to R^{\oplus d}\to M\to0$, and localizing this at the prime ideal $\p=xR$ yields $R_\p^{\oplus d}\cong(\syz M)_\p\cong R_\p^{\oplus a}\oplus R_\p^{\oplus b}$ as $M_\p=0=(R/yR)_\p$ and $(R/xR)_\p=R_\p$.
This implies $d=a+b$.
Similarly, localization at the prime ideal $\q=yR$ shows $d=a+c$.
Therefore we have $b=c$, and $\syz M\cong R^{\oplus a}\oplus(R/xR\oplus R/yR)^{\oplus b}\cong R^{\oplus a}\oplus\m^{\oplus b}$.

(2) Let $E$ be a CM $R$-module of codimension $1$ and $f:X\to E$ a MCM approximation.
Then $E$ is of finite length, and stably equivalent to $\syz^{-1}\syz E$ by Remark \ref{mcmappr}.
It follows from (1) that $\syz E$ is stably equivalent to $\m^{\oplus m}$ for some $m\ge0$.
We easily see $\syz\m\cong\m$, whence $\syz^{-1}\m\cong\m$.
Therefore $X$ is stably equivalent to $\m^{\oplus m}$.
Thus one can show the assertion similarly to the proof of Proposition \ref{wa}(1).
\end{proof}

\begin{rem}
In view of (the proof of) Example \ref{xy}(1), one might wonder if every Gorenstein local ring having exactly two indecomposable nonfree MCM modules $M,N$ is such that if the $r$th syzygy of a CM module of codimention $r$ contains $M$ as a direct summand, then it also contains $N$ as a direct summand.
This is not true in general.
Indeed, consider the Gorenstein local ring $R=k[[x,y]]/(x^2-y^5)$ with $k$ an algebraically closed field.
Then by \cite[Proposition (5.11)]{Y} the ring $R$ has only three nonisomorphic indecomposable MCM modules: $R$, $\m=(x,y)$ and $\n=(x,y^2)$.
For a MCM module $M$ and a CM module $C$ of codimension $r$, let $\chi_M(C)$ denote the number of copies of $M$ appearing as a direct summand in the MCM module $\syz^rC$.
Then $X:=R/\m$ and $Y:=R/\n$ are CM $R$-modules of codimension $1$.
One has $\chi_\m(X)=1\ne0=\chi_\n(X)$ and $\chi_\m(Y)=0\ne1=\chi_\n(Y)$.
\end{rem}

Now we prove the main result of this section, which gives a characterization of the perfect linkage in codimension one.

\begin{thm}\label{main2}
Let $R$ be a Gorenstein local ring of positive dimension.
Let $M,N$ be CM $R$-modules of codimension one.
\begin{enumerate}[\rm(1)]
\item
$M$ is doubly perfectly linked to $N$ if and only if $\syz M\scong\syz N$.
\item
$M$ is triply perfectly linked to $N$ if and only if $(\syz M)^*\scong\syz N$.
\item
If $M$ and $N$ are perfectly linked, then they are either doubly or triply perfectly linked.
\end{enumerate}
\end{thm}

\begin{proof}
(1) The assertion is shown by Theorem \ref{main}.

(2) If follows from Corollary \ref{genyi} that for any perfect morphism $f$ of $M$ we have
\begin{equation}\label{sysysy}
\syz\li_fM\scong\syz^2\tr\syz M\scong(\syz M)^*.
\end{equation}
Suppose that $M$ is triply perfectly linked to $N$.
Then there exist CM modules $X$ and $Y$ such that $M\sim X\sim Y\sim N$.
Hence $X$ is doubly perfectly linked to $N$, and $\syz X$ is stably equivalent to $\syz N$ by (1).
There exists a perfect morphism $f$ of $M$ such that $\li_fM=X$, and we obtain $\syz\li_fM=\syz X\scong\syz N$.
Therefore $(\syz M)^*$ is stably equivalent to $\syz N$ by \eqref{sysysy}.

Conversely, assume $(\syz M)^*\scong\syz N$.
Then by using \eqref{sysysy} we see that $\syz\li_fM$ is stably equivalent to $\syz N$ for all perfect morphisms $f$ of $M$.
So now take a perfect morphism $f:P\to M$ (see Remark \ref{doko}).
Then we have $\syz\li_fM\scong\syz N$, and Theorem \ref{main} implies that $\li_g\li_fM\cong\li_hN$ for some perfect morphisms $g,h$.
Therefore $M$ is triply perfectly linked to $N$.

(3) Suppose that there exist CM modules of codimension one, $M_0,\dots,M_n$ with $n\ge1$, such that $M=M_0\sim M_1\sim\cdots\sim M_n=N$.

First, we consider the case $n=1$.
In this case $M$ is directly perfectly linked to $N$.
By Proposition \ref{basli}(3) and Remark \ref{doko}, $N$ is doubly perfectly linked to $N$ itself.
Thus $M$ is triply perfectly linked to $N$.

Let $n\ge2$.
For each $0\le i\le n-2$ the module $M_i$ is doubly perfectly linked to $M_{i+2}$, and it follows from (1) that $\syz M_i\scong\syz M_{i+2}$.
Hence if $n$ is even, then we have
$$
\syz M=\syz M_0\scong\syz M_2\scong\syz M_4\scong\cdots\scong\syz M_n=\syz N.
$$
Applying (1) again shows that $M$ is doubly perfectly linked to $N$.
If $n$ is odd, then $n-3$ is even, and $\syz M_3$ is stably equivalent to $\syz M_n=\syz N$.
Since $M=M_0$ is triply perfectly linked to $M_3$, by (2) the module $(\syz M)^*$ is stably equivalent to $\syz M_3$.
Hence we get $(\syz M)^*\scong\syz N$, and $M$ is triply perfectly linked to $N$ by (2) again.
\end{proof}

Taking advantage of our Theorem \ref{main2}, we obtain the following example.
This should be remarkable in that it provides a pair of ideals that are not CI-linked.

\begin{ex}
Let $R=k[[x,y]]/(x^2-y^4)$, where $k$ is a field of characteristic not two.
Let $\m=(x,y)$ and $I=(x,y^2)$.
Then the CM $R$-modules $R/\m$ and $R/I$ are not perfectly linked.
In particular, the ideals $\m$ and $I$ are not CI-linked.
\end{ex}

\begin{proof}
Note that $I$ is an $\m$-primary ideal, and thus both $R/\m$ and $R/I$ are CM $R$-modules of codimension $1$.
The last assertion follows from Proposition \ref{liid}(2), so let us show the first assertion.
Let $\p=(x+y^2)$ and $\q=(x-y^2)$ be prime ideals of $R$.
Then $\p$ and $\q$ do not meet, and $I$ is isomorphic to $\p\oplus\q$ as an $R$-module, for $\ch k\ne2$.
Since $\p,\q,\m$ are nonfree indecomposable MCM $R$-modules, $\syz(R/I)=I$ is not stably equivalent to $\m=\syz(R/\m)$.
Note that $\p\cong R/\q$ and $(R/\q)^*\cong(0:\q)=\p$.
Hence $\p^*\cong\p$, and similarly, $\q^*\cong\q$.
Therefore we have $(\syz(R/I))^*=I^*\cong I$, which is not stably equivalent to $\syz(R/\m)$.
Finally, applying Theorem \ref{main2} yields that $R/\m$ and $R/I$ are not perfectly linked.
\end{proof}

\section*{Acknowledgments}
The authors thank the referee for giving them several useful comments and suggestions.


\end{document}